\documentclass[fontsize=12pt,parskip=half]{scrartcl}
\usepackage{amsmath,amssymb,amsfonts,amsthm,mathabx}
\usepackage[dvips]{graphics}
\usepackage{enumerate}
\usepackage{mathrsfs}
\usepackage{subfigure}
\usepackage{color}

\usepackage[T1]{fontenc}

\DeclareSymbolFont{AMSb}{U}{msb}{m}{n}

\def\B{\mathcal {B}}

\def\P{\mathbb P}
\def\N{\mathbb N}

\def\K{\mathcal {K}}
\def\M{\mathcal {M}}
\def\V{\mathcal {V}}
\def\R{\mathbb {R}}
\def\Z{\mathbb {Z}}
\def\dd{\mathrm {d}}

\def\FF{\mathscr{F}}
\def\EE{\mathcal {E}}
\def\E{\mathbb {E}}
\def\T{\mathbb {T}}
\def\1{\,{\makebox[0pt][c]{\normalfont    1}
\makebox[2.5pt][c]{\raisebox{3.5pt}{\tiny {$\|$}}}
\makebox[-2.5pt][c]{\raisebox{1.7pt}{\tiny {$\|$}}}
\makebox[2.5pt][c]{} }}

\def\eps{\varepsilon}
\def\scirc{\mathbin{\raise.15ex\hbox{\scriptsize$\circ$}}}

\newtheorem{thm}{Theorem}
\newtheorem{prop}[thm]{Proposition}

\newtheorem{lemma}[thm]{Lemma}

\theoremstyle{definition}
\newtheorem{defn}[thm]{Definition}
\newtheorem{bem}[thm]{Remark}
\newtheorem{bemn}[thm]{Remarks}
\newtheorem{bsp}[thm]{Example}

\renewenvironment{proof}{\noindent\textbf{Proof.}\ }{\qed}

\title{Minimal random attractors}

\author{Hans Crauel%
  \thanks{Institut f\"ur Mathematik, 60629 Frankfurt, FRG; \ 
    \small\texttt{crauel{\scriptsize @}math.uni-frankfurt.de}}
\and Michael Scheutzow%
  \thanks{Institut f\"ur Mathematik, MA 7-5, Fakult\"at II, 
    Technische Universit\"at Berlin, 
    Stra\ss e des 17.~Juni 136, 10623 Berlin, FRG;  \ 
    \small\texttt{ms{\scriptsize @}math.tu-berlin.de}}}

\date{}

\begin{document}\maketitle

\begin{abstract}\noindent
  It is well-known that random attractors of a random dynamical 
  system are generally not unique. 
  We show that for general pullback attractors and weak attractors, 
  there is always a minimal (in the sense of smallest) random 
  attractor which attracts a given family of (possibly random) sets. 
  We provide an example which shows that this property need not hold 
  for forward attractors.  
  We point out that our concept of a random attractor is very general: 
  The family of sets which are attracted is allowed to be completely 
  arbitrary.   
\par\medskip

  \ Secondary\, 
  34D55 \ 35B51 \ 37B25 \ 37C70 \ 37L99 \ 60H10 \ 60H15 \ 60H25 
\end{abstract}


\section{Introduction}
For deterministic dynamical systems on metric spaces the notion of 
an attractor is well established. 
The most common notion, mainly used for partial differential equations 
(PDEs) on suitable Hilbert or Banach spaces, is that of a \emph{global 
set attractor}. 
It is characterized by being a compact set, being strictly invariant, 
and attracting every compact, or even every bounded set. 
Uniqueness of the global set attractor is immediate.   \par

Another very common notion is that of a \emph{(global) point attractor}, 
which is often used for systems on locally compact spaces (which are 
often Euclidean spaces or finite-dimensional manifolds). 
A global point attractor is again compact and strictly invariant, but 
it is only assumed to attract every point (or, equivalently, every 
finite set). 
The global point attractor is in general not unique, which can be seen 
from the simple dynamical system induced by the scalar differential 
equation $\dot x=x-x^3$. 
Here the unique global set attractor is the interval $[-1,1]$, which, 
of course, is also a point attractor. 
But also $[-1,0]\cup\{1\}$, and $\{-1\}\cup[0,1]$, are point attractors. 
Obviously, there is a minimal point attractor, namely $\{-1,0,1\}$.  \par

In fact, for deterministic systems the following result is well known: 
Whenever $\mathcal B$ is an arbitrary family of non-empty subsets of 
the state space then there exists an attractor for $\mathcal B$ (i.e., 
a compact, strictly invariant set attracting every $B\in\mathcal B$) 
if and only if there exists a compact set such that every 
$B\in\mathcal B$ is attracted by this set. 
Furthermore, there exists a unique minimal attractor for~$\mathcal B$, 
which is given by the closure of the union of the $\omega$-limit sets 
of all elements of~$\mathcal B$. 
This minimal attractor for~$\mathcal B$ is addressed as \emph{the} 
$\mathcal B$-attractor. 
If a global set attractor exists then whenever a $\mathcal B$-attractor 
exists it is always a subset of the global set attractor.  \par

For random dynamical systems an analogous statement has been established 
in~\cite{crauel01}. 
However, this result had to assume a separability condition 
for~$\mathcal B$.  \par

The aim of the present paper is to remove this separability condition, 
i.e.\ to establish that for any family~$\mathcal B$ of (possibly 
even random) sets for which there is a compact random set attracting 
every element of~$\mathcal B$, there exists a unique minimal random 
attractor for~$\mathcal B$.  \\
Furthermore, it is in general not true that this $\mathcal B$-attractor 
is given by the closed random set 
\begin{equation}  \label{tq7}
  \overline{\bigcup_{B\in\mathcal B}\Omega_B(\omega)}\quad\mbox{almost surely}, 
\end{equation}
where $\Omega_B(\omega)$ denotes the (random) $\Omega$-limit set of~$B$. 
This is shown using an example of a random dynamical system induced by 
a stochastic differential equation on the unit circle~$S^1$. 
Here the global set attractor is the whole~$S^1$ (which is a strictly 
invariant compact set). 
If~$\mathcal B$ is taken to be the family of all deterministic points 
in~$S^1$ it is shown that~\eqref{tq7} gives~$S^1$, the global set attractor, 
almost surely. 
However, the minimal point attractor is a one point set consisting of 
a random variable, which (pullback) attracts every solution starting 
in a deterministic point.

\section{Notation and Preliminaries}  \label{sq2}
Let~$E$ be a Polish space, i.e.\ a separable topological space whose 
topology is metrizable with a complete metric. 
Several assertions in the following are formulated in terms of a 
metric~$d$ on~$E$ which is referred to without further mentioning. 
This metric will always be assumed to generate the topology of~$E$ 
and to be complete, even if some of the assertions hold also if~$d$ 
is not complete. 
For $x\in E$ and $A\subset E$ we define 
$d(x,A)=\inf\{d(x,a):\,a\in A\}$ with the convention 
$d(x,\emptyset)=\infty$. 
For non-empty subsets~$A$ and~$B$ of~$E$ we denote the Hausdorff 
semi-metric by $d(B,A):=\sup\{d(b,A):\,b\in B\}$ and 
define $d(\emptyset,A):=0$ and $d(B,\emptyset):=\infty$ in case 
$B\neq\emptyset$.  \\
We note that with this convention both for the empty family 
$\mathcal B=\emptyset$ as well as for $\mathcal B=\{\emptyset\}$ 
the empty set $A=\emptyset$ is an attractor, in fact the minimal 
one.  \par

We denote the Borel $\sigma$-algebra on~$E$ (i.e.\ the smallest 
$\sigma$-algebra on $E$ which contains every open set) by $\EE$.  \par

Suppose that $(\Omega,\mathscr F,\P)$ is a probability space, 
$\T_1 \in \{\Z,\R\}$, and 
\begin{align*}
  \vartheta:\mathbb \T_1\times\Omega&\to\Omega  \\
  (t,\omega)&\mapsto\vartheta_t\omega
\end{align*}
is a measurable map, such that $\vartheta_t:\Omega\to\Omega$ 
preserves~$\P$, and such that 
$\vartheta_{t+s}=\vartheta_t\scirc\vartheta_s$ for all 
$s,t\in\mathbb \T_1$ and $\vartheta_0={\mathrm{id}}$. 
Thus $(\vartheta_t)$ is a classical measurable dynamical system 
on $(\Omega,\mathscr F,\P)$.

\begin{defn}  \label{tq1}%
Given $(\Omega,\mathscr F,\P)$ and $\vartheta_t,\,t \in \T_1$ as above, 
$E$ a Polish space, and~$\T_2$ either~$\R$, $[0,\infty)$, 
$\Z$, or~$\N_0$ such that $\T_2 \subseteq \T_1$, a measurable map 
\begin{eqnarray*}
  \varphi:\T_2\times E\times\Omega&\to&E  \\
  (t,x,\omega)&\mapsto&\varphi(t,\omega)x
\end{eqnarray*}
(or the pair $(\varphi,\vartheta)$) is a \emph{random dynamical system 
(RDS)} on~$E$ if 
\begin{enumerate}[(i)]
\item $\varphi(t,\omega):E\to E$ is continuous for every $t\in \T_2$, 
  $\P$-almost surely 
\item
  $\varphi(t+s,\omega)
  =\varphi(t,\vartheta_s\omega)\scirc\varphi(s,\omega)$ 
  for all $s,t\in \T_2$, and $\varphi(0,\omega)={\mathrm{id}}$, 
  for~$\P$-almost all~$\omega$.
\end{enumerate}
\end{defn}

\begin{bemn}\label{wichtig}
\begin{enumerate}[(i)]
\item Note that we do not assume continuity in~$t$ here.  
\item An RDS $(\varphi,\vartheta)$ is said to be \emph{two-sided} if~$\T_2$ 
  is two-sided. 
  For a two-sided RDS~$\varphi$ the maps $\varphi(t,\omega)$ are 
  invertible, and $\varphi(t,\omega)^{-1}=\varphi(-t,\vartheta_t\omega)$ 
  a.\,s.
\item Proposition~\ref{perf1} shows that one can always change~$\varphi$ 
  on a set of measure 0 in such a way that properties~(i) and~(ii) in 
  Definition~\ref{tq1} hold without exceptional sets. 
  In the following we will tacitly assume that~$\varphi$ satisfies these 
  slightly stronger assumptions. 
  The question whether exceptional sets of measure zero in~(ii) of 
  Definition~\ref{tq1} which may depend on~$s$ and~$t$ can be eliminated 
  (without destroying possible (right-)\-continuity properties of~$\varphi$ 
  in the time variable) has been addressed in~\cite{arnold}, 
  \cite{arnold-scheutzow}, \cite{Scheutzow96}, and \cite{Kager-Scheutzow}. 
\end{enumerate}
\end{bemn}

\begin{defn}
  Let $(\Omega,\mathscr F,\mathbb P)$ be a probability space and~$E$ a 
  Polish space. 
  A \emph{random set}~$C$ is a measurable subset of $E\times\Omega$ 
  (with respect to the product $\sigma$-algebra $\EE\otimes\mathscr F$). 
 \end{defn}

 The $\omega$-section of a set~$C\subset E \times \Omega$ is defined by 
 \begin{equation*}
   C(\omega)=\{x:(x,\omega)\in C\},\quad\omega\in\Omega. 
 \end{equation*}
In the case that a set~$C\subset E\times\Omega$ has closed or 
compact $\omega$-sections it is a random set as soon as the mapping 
$\omega\mapsto d\bigl(x,C(\omega)\bigr)$ is measurable (from~$\Omega$ 
to $[0,\infty)$) for every $x\in E$, see~\cite[Chapter~2]{crauel02}. 
Then~$C$ will be said to be a \emph{closed} or a \emph{compact}, 
respectively, random set. For any set~$C\subset E\times\Omega$, we define
$\overline{C}:=\{(x,\omega):\,x\in \overline{C(\omega)}\}$

\begin{bem}
  Note that it does not suffice to define random sets by demanding 
  $\omega\mapsto d\bigl(x,C(\omega)\bigr)$ to be measurable 
  for every $x\in E$. 
  In this case the associated $\{(x,\omega):x\in C(\omega)\}$ need 
  not be an element of $\EE \otimes \mathscr F$, 
  see~\cite[Remark~4]{crauel-kloeden15}. 
\end{bem}

\begin{defn}
If~$\varphi$ is an RDS then a set $D\subset E \times \Omega$ 
is said to be \emph{forward invariant} or \emph{strictly invariant}, 
respectively, with respect to~$\varphi$, if 
$\varphi(t,\omega)D(\omega)\subset D(\vartheta_t\omega)$ or 
$\varphi(t,\omega)D(\omega)=D(\vartheta_t\omega)$ $\P$-a.\,s., 
respectively, for every $t\geq0$.  
\end{defn}

\begin{defn}
For $B\subset E\times\Omega$ we define the 
\emph{$\Omega$-limit set of $B$} by 
\begin{equation*}
  \Omega_B(\omega)=\bigcap_{T\geq0}
  \overline{\bigcup_{t\geq T}
    \varphi(t,\vartheta_{-t}\omega)(B(\vartheta_{-t}\omega))}, 
  \quad \omega\in\Omega, 
\end{equation*}
as in~\cite[Definition 3.4]{crauel99}. 
\end{defn}

\begin{bem}
  It is easy to verify that an $\Omega$-limit set is 
  always forward invariant and also that it is strictly   
  invariant for an RDS with two-sided time.  

  The following Lemma, which generalizes \cite[Theorem~3.4]{crauel01}, 
  provides another sufficient condition for $\Omega_B$ to be strictly 
  invariant. 
\end{bem}

\begin{lemma}  \label{tq2}
Suppose that~$B\subset E\times\Omega$, $K\subset E\times\Omega$, and 
$\Omega_0 \subset \Omega$. 
Assume that for all $\omega \in \Omega_0$, the set $K(\omega)$ is 
compact and  
\begin{equation}  \label{tq20}
\lim_{t\to\infty}
  d\bigl(\varphi(t,\vartheta_{-t}\omega)B(\vartheta_{-t}\omega),
         K(\omega)\bigr)=0.  
\end{equation}
Then, for every $t\geq0$ and $\omega \in \Omega_0$, 
$\Omega_B(\vartheta_t\omega)\subset\varphi(t,\omega)\Omega_B(\omega)$ 
and $\Omega_B(\omega)\subset K(\omega)$. 
If, moreover, $\Omega_0\in\mathscr{F}$ and $\P(\Omega_0)=1$, 
then~$\Omega_B$ is strictly invariant.
\end{lemma}

\begin{proof}
Fix $\omega\in\Omega_0$ throughout the proof. 
Compactness of $K(\omega)$ and~\eqref{tq20} imply 
$\Omega_B(\omega)\subset K(\omega)$. 
Suppose that $y\in\Omega_B(\vartheta_t\omega)$ for some $t\geq0$. 
Then 
$y=\lim_{n\to\infty}\varphi(t_n,\vartheta_{-t_n}(\vartheta_t\omega))b_n$ 
for sequences $t_n\to\infty$ and $b_n\in B(\vartheta_{t-t_n}\omega)$. 
Consider the sequence 
$\varphi(t_n-t,\vartheta_{-(t_n-t)}\omega)b_n$, defined for~$n$ 
with $t_n-t\geq0$. 
By~\eqref{tq20} we have 
$\lim_{n \to \infty} 
  d\bigl(\varphi(t_n-t,\vartheta_{-(t_n-t)}\omega)b_n,K(\omega)\bigr)
  =0$ 
for $n\to\infty$. 
Compactness of~$K(\omega)$ implies that this sequence has a convergent 
subsequence. 
Choose one and denote its limit by $z(\omega)$, then  
$z(\omega)\in\Omega_B(\omega)$. 
Using the same notation for the subsequence, continuity of 
$\varphi(t,\omega)$ implies 
\begin{equation*}
  \varphi(t,\omega)z(\omega)
  =\lim_{n\to\infty}\varphi(t_n,\vartheta_{-(t_n-t)}\omega)b_n
  =y.  
\end{equation*}
Thus, for any $y\in\Omega_B(\vartheta_t\omega)$ there exists 
$z\in\Omega_B(\omega)$ with $\varphi(t,\omega)z=y$, whence 
$\Omega_B(\vartheta_t\omega)\subset\varphi(t,\omega)\Omega_B(\omega)$.

The final statement of Lemma~\ref{tq2} follows since $\Omega_B$ is 
forward invariant. 
\end{proof}

Now let~$\mathcal B$ be a non-empty family of sets 
$B\subset E\times\Omega$. 
At this point we make no measurability assumptions on the sets in~$\B$ 
and therefore say that a property depending on~$\omega$ holds 
\emph{almost surely} if there is a measurable set of full measure on 
which the property holds. 
Analogously, we will interpret statements like ``$Y\to 0$ in 
probability'' for a real-valued (possibly non-measurable) function 
$\omega\mapsto Y(\omega)$. 
As usual, we introduce the \emph{universal completion}~$\mathscr F^{\mathrm u}$ 
of~$\mathscr F$ as the intersection of all completions of~$\FF$ with 
respect to probability measures on~$\FF$. 
Note that one automatically has $\FF^{\mathrm u}=\FF$ in the case 
of a complete probability space~$(\Omega,\FF,\P)$.   \par
 
We define the concept of a random pullback, forward and weak $\B$-attractor 
of an RDS $(\varphi,\vartheta)$ as usual (except that we do not impose any 
measurability assumptions on~$\B$). 
Random pullback attractors were first introduced in~\cite{crauel-flandoli} 
while the concept of a weak attractor is due to G.~Ochs~\cite{ochs}. 

\begin{defn}\label{tq3} 
Suppose that~$\varphi$ is an RDS on a Polish space~$E$ and~$\mathcal B$ 
is a non-empty family of subsets of $E\times\Omega$. 
Then a set $A\subset E\times\Omega$ is a \emph{random attractor 
for~$\mathcal B$} if
\begin{enumerate}[(i)]
\item $A$ is a compact random set
\item $A$ is strictly $\varphi$-invariant, i.e.\ 
  \begin{equation*}
    \varphi(t,\omega)A(\omega)=A(\vartheta_t\omega)
  \end{equation*}
   $\P$-almost surely for every $t\in \T_2$ with $t\geq0$ 
\item $A$ attracts~$\mathcal B$, i.e. 
  \begin{equation}  \label{tq98}
    \lim_{t\to\infty}d\bigl(\varphi(t,\vartheta_{-t}\omega)B(\vartheta_{-t}\omega), 
    A(\omega)\bigr)=0\qquad \P\mbox{-a.s.}
  \end{equation}
  for every $B\in\mathcal B$; 
  in this case~$A$ is a \emph{random pullback attractor for~$\mathcal B$},  \\
  or 
  \begin{equation*} 
    \lim_{t\to\infty}
    d\bigl(\varphi(t,\omega)B(\omega),A(\vartheta_t\omega)\bigr)=0\qquad\P\mbox{-a.s.},
  \end{equation*}
  for every $B\in\mathcal B$; 
  in this case~$A$ is a \emph{random forward attractor for~$\mathcal B$}, \\ 
  or
  \begin{equation}  \label{tq11}
    \lim_{t\to\infty}d\bigl(\varphi(t,\omega)B(\omega), 
    A(\vartheta_t\omega)\bigr)=0\qquad \mbox{in probability}
  \end{equation}
  for every $B\in\mathcal B$; 
  in this case~$A$ is a \emph{weak (random) attractor for~$\mathcal B$}. 
\end{enumerate}
\end{defn}

\begin{bem}
  Note that the property of being a pullback or weak attractor does 
  not depend on the choice of the metric~$d$ metrizing the topology 
  of~$E$. 
  This is not true for forward attactors, not even when $E=\R$, 
  see~\cite[Example 2.4]{Flandoli-Gess-Scheutzow1} or the example 
  in Section~\ref{non}. 
\end{bem}

\begin{bemn}\label{rem3}
  (i) \ For a weak random attractor condition~\eqref{tq11} is equivalent 
  to~\eqref{tq98} with almost sure convergence replaced by convergence 
  in probability. 
  See~\cite{crauel-dimitroff-scheutzow} for sufficient conditions 
  for the existence of weak global set attractors.  For \emph{monotone} 
  RDS, the concept of a weak attractor turns out to be more suitable 
  than that of a pullback attractor, see~\cite{CS04} 
  and~\cite{Flandoli-Gess-Scheutzow2}.  \par

  (ii) \ While for non-autonomous systems it is very simple to find 
  examples of attractors which are pullback but not forward, and vice 
  versa (see, e.g., \cite{crauel-kloeden15}), this is not so simple 
  for random attractors. 
  One of the authors~\cite{Scheutzow02} has constructed examples for this 
  to happen as well as examples where only weak random attractors 
  exist which are neither pullback nor forward attractors.  \par

  (iii) Clearly each random pullback attractor~$A$ for~$\B$ must satisfy 
  $\Omega_B(\omega)\subset A(\omega)$ almost surely for every $B\in\B$ 
  (the exceptional sets may depend on~$B$).  \par

  (iv) Instead of assuming an attractor~$A$ to be a compact random set 
  it suffices to assume~$A$ to be a random set such that $A(\omega)$ is 
  compact~$\mathbb P$-almost surely, which is slighty weaker 
  (see~\cite[Proposition~2.4]{crauel02} for details). 
  Proposition~\ref{lemma} below, applied to the singleton~$A$, implies 
  existence of a compact random set satisfying the conditions of 
  Definition~\ref{tq3}. 
\end{bemn}

\begin{bem}
  For a general family~$\mathcal B$ there may or may not exist a 
  (pullback, forward or weak) random $\B$-attractor for an RDS~$\varphi$ 
  and if such an attractor exists then it need not be unique in general.   
  However, as soon as~$\mathcal B$ contains every compact deterministic 
  set, then whenever a weak random attractor for~$\mathcal B$ exists then 
  it is unique, see~\cite[Lemma~1.3]{Flandoli-Gess-Scheutzow1}. 
  Since every pullback and every forward attractor is also a weak 
  attractor, the same uniqueness statement holds for pullback and 
  forward attractors (for pullback attractors this was already 
  established in~\cite{crauel99}). 
  
  Whenever a $\B$-attractor (in whatever sense) is not unique it 
  is natural to ask whether there is a smallest (or minimal) 
  $\B$-attractor. 
  In~\cite[Theorem~3.4]{crauel01} a condition on~$\mathcal B$ (formulated 
  only for families of deterministic sets) for the existence of a minimal 
  random pullback attractor has been given.  It is one of the aims of 
  this paper to show that a minimal random pullback (respectively weak) 
  attractor exists for arbitrary~$\mathcal B$, provided existence of at 
  least one pullback (respectively weak) $\mathcal B$-attractor.  
\end{bem}

\section{Existence of a minimal pullback attractor for~$\B$}

Given a random dynamical system $(\varphi,\vartheta)$ on the space 
$(\Omega,\FF,\P)$ taking values in the Polish space~$(E,d)$ we 
consider a general family~$\B$ of subsets of $E\times\Omega$ for 
which we assume existence of a pullback attractor or, equivalently, 
the existence of an attracting compact random set. 
Note that no measurability conditions whatsoever are imposed on 
(the elements of)~$\B$. 
The major result of this section is the existence of a minimal random 
pullback attractor for~$\B$. 
Of course this is immediate as soon as~$\B$ contains every compact 
deterministic subset of~$E$ (or, more precisely, every $C\times\Omega$ 
with $C\subset E$ compact); in this case the minimal attractor is 
the unique global set attractor. 
However, for instance for~$\B$ consisting of all (deterministic) 
points (or, equivalently, of all finite sets), there may be several 
random pullback attractors.

\begin{thm}\label{propo} 
  Let~$\B$ be an arbitrary family of sets $B \subset E \times \Omega$ and 
  let $(\varphi,\vartheta)$ be an RDS on~$E$. 
\begin{enumerate}[(i)]
\item Assume that there exists a random set $\K \subset E \times \Omega$ 
  such that $\K(\omega)$ is $\P$-a.s.\ compact and $\K$ attracts $\B$, 
  i.e.\ 
  \begin{equation*}
    \lim_{t\to\infty}d\bigl(\varphi(t,\vartheta_{-t}\omega)B(\vartheta_{-t}\omega), 
    \K(\omega)\bigr)=0\qquad \P\mbox{-a.s.}
  \end{equation*}
  (and therefore $\Omega_B(\omega) \subset \K(\omega)$ a.s.) for every 
  $B\in\B$. 
  Then $(\varphi,\vartheta)$ has at least one pullback attractor for~$\B$. 
\item If $(\varphi,\vartheta)$ has at least one pullback attractor 
  for~$\B$, then the RDS has a \emph{minimal} pullback $\B$-attractor~$A$. 
  In addition, there is a countable sub-family~$\B_0\subset\B$ such 
  that~$A$ is also the minimal pullback $\B_0$-attractor. 
\item Under the assumptions of (i) the minimal pullback $\B$-attractor~$A$ 
  is almost surely contained in $\K(\omega)$. 
\end{enumerate}
\end{thm}

Note that under the assumptions of~(ii) any pullback attractor~$\K$ 
for~$\B$ will satisfy the assumptions of~(i). 
Therefore, it suffices to show that under the assumptions of~(i) there 
exists a minimal pullback attractor for~$\B$ and that it satisfies 
the conclusions of~(ii) and~(iii).  \par

For the proof of the theorem we are going to prepare several results which 
will also be used in the next section, where we investigate the same 
question for weak attractors. 
For the first few results we just need an arbitrary probability space 
$(\Omega,\FF,\P)$ and a Polish metric space $(E,d)$, but no RDS\@. 
We will denote the open ball with centre $x\in E$ and radius $r>0$ 
by~$B(x,r)$. 
By~$D$ we always denote a fixed countable dense set in~$E$. 
The following lemma just requires a separable metric (not necessarily 
Polish) space $(E,d)$. 
It is a straightforward consequence of classical results in topology. 

\begin{lemma}\label{cover}
  For each fixed $r>0$ there exists a cover $R(x,r)$, $x\in D$, of~$E$ such 
  that $R(x,r)$ is a closed (possibly empty) subset of $B(x,r)$, and for each 
  $y \in E$ there exists a neighbourhood of~$y$ which intersects only finitely 
  many of the sets $R(x,r)$, $x\in D$, and there exists $x\in D$ such 
  that~$y$ is in the interior of $R(x,r)$. 
\end{lemma}

\begin{proof}
  For given $r>0$, the family $B(x,r)$, $x\in D$, is an open cover of~$E$. 
  Since every metric space is paracompact, 
  Theorem  VIII.4.2 in  Dugundji~\cite{Du76}
  implies the existence of a cover $R(x,r)$, $x\in D$, as claimed in the 
  lemma. 
\end{proof}

\begin{bem}\label{coverbem}
  Note that for each $r>0$ the closed cover $R(x,r)$, $x\in D$, given in 
  Lemma~\ref{cover}, has the property that for every (possibly infinite) 
  subset $D_0\subset D$ the union $\bigcup_{x \in D_0}R(x,r)$ 
  is closed. 
  This will be important in the following. 
\end{bem}

\begin{bem}\label{coverbem2}
Note that for each $r>0$, the family $\mathring R(x,r),\,x\in D$, is 
an open cover of $E$ thanks to the last assertion of Lemma~\ref{cover} 
(here~$\mathring S$ denotes the interior of the set~$S$). 
\end{bem}

In the following $R(x,r)$, $x\in D$, will always denote a closed 
cover as in Lemma~\ref{cover} (which, of course, is not unique). 
The following result asserts that every subset of~$E\times\Omega$ 
has a \emph{closed random hull}.  

\begin{prop}\label{singleset}
  Let $K\subset E\times\Omega$. 
  There exists a (unique) smallest closed random set $\hat K$ which 
  contains~$K$ in the following sense: $K \subset \hat K$ and for 
  every random set~$S$ for which $K(\omega)\subset S(\omega)$ for 
  almost all $\omega\in\Omega$ and for which $S(\omega)$ is closed 
  for almost all $\omega\in\Omega$, we have 
  $\hat K(\omega)\subset S(\omega)$ for almost all $\omega \in \Omega$. 
\end{prop}

\begin{proof}
  For $G \in \EE$ define 
  \begin{equation*}
    M^{G}:=\{\omega: K(\omega)\cap G \neq \emptyset\}.  
  \end{equation*}
  This set will not be measurable in general. 
  Let $\beta_G:=\P^*(M^{G}):=\inf\{\P(M): M \in \FF, M \supset M^{G}\}$ 
  and let $M_1 \supset M_2 \supset \ldots$ be sets in $\FF$ with 
  $M^G\subset M_i$, $i\in\mathbb N$, such that 
  $\lim_{i\to\infty}\P(M_i)=\beta_G$. 
  Define $\tilde M^{G}:=\bigcap_{i=1}^\infty M_i$ and 
  \begin{equation}\label{hatK}
    \hat K:=\bigcap_{n=1}^\infty
            \bigcup_{x \in D}
            \bigl(R(x,{\textstyle\frac 1n})\times \tilde M^{R(x,\frac 1n)}\bigr).
\end{equation}
$\hat K$ is a random set. 
Using Lemma~\ref{cover} and Remark~\ref{coverbem}, we see that $\hat K(\omega)$ 
is closed for all $\omega \in \Omega$. 
Since $\omega\mapsto d\bigl(y,(R(x,\frac1n)\times \tilde M^{R(x,\frac 1n)})(\omega)\bigr)$ 
is measurable for each $x$, $y$ and $n$, the same is true for the map 
$\omega\mapsto d(y,\hat K(\omega))$, so that~$\hat K$ is a closed random set. 
By construction, we have $K(\omega)\subset\hat K(\omega)$ for all $\omega\in\Omega$. 
Indeed, 
\begin{equation*}
  \overline{K}=\bigcap_{n=1}^\infty \bigcup_{x \in D}
  \bigl(R(x,{\textstyle\frac1n})\times M^{R(x,\frac 1n)}\bigr). 
\end{equation*}
It remains to show the minimality property of~$\hat K$.  \par

Let~$S$ be a set as in the statement of the proposition and let $G\in \EE$. 
Then, almost surely, 
\begin{equation*}
  K(\omega)\cap G \subset S(\omega) \cap G  
\end{equation*}
and therefore $S(\omega) \cap G \neq \emptyset$ for almost every $\omega\in M^{G}$. 
Since~$E$ is Polish, the projection theorem (see, e.g., Dellacherie and 
Meyer~\cite[III.44]{DM78}) implies that the set 
$\{\omega\in\Omega: \, S(\omega)\cap G \neq \emptyset\}$ is in $\FF^\mathrm u$ and 
therefore $S(\omega) \cap G \neq \emptyset$ for almost all $\omega\in\tilde M^{G}$. 
Let $\Omega_0 \in \FF$ be a set of full measure such that $S(\omega)$ is closed 
for all $\omega \in \Omega_0$ and  $S(\omega)\cap G\neq\emptyset$ for all 
$\omega\in \tilde M^{G}\cap \Omega_0$ and all $G=R(x,\frac1n)$, $x\in D$, $n\in\N$. 
Let $\omega \in \Omega_0$ and $y \in \hat K(\omega)$. 
By definition of~$\hat K$ this means that for every $n \in \N$ there exists some 
$x \in D$ such that $y \in R(x,\frac1n)$ and $\omega \in \tilde M^{R(x,\frac1n)}$, 
so $d(y,S(\omega))\le 2/n$. 
Since~$S(\omega)$ is closed this implies $y \in S(\omega)$, so the proof of the 
proposition is complete. 
\end{proof}

\begin{bem}
The assertion of Proposition~\ref{singleset} is wrong without the word 
\emph{closed} (at both places where it appears). 
As an example consider the deterministic case in which~$\Omega$ is a 
singleton and let~$K$ be a set which is not in~$\EE$. 
There exists no smallest measurable subset of~$E$ which contains~$K$.
\end{bem}

\begin{prop}\label{lemma}
  Assume that $K_\alpha$, $\alpha\in I$, is a family of sets in 
  $\EE\otimes\mathscr F^\mathrm u$. 
  Then there exists a closed random set~$A$ 
  such that for every $\alpha\in I$ we have $K_\alpha(\omega)\subset A(\omega)$ 
  almost surely and~$A$ is the minimal set with this property: 
  $A(\omega) \subset S(\omega)$ almost surely 
  for every $S \subset E \times \Omega$ for which $S(\omega)$ is almost 
  surely closed and for which $K_\alpha(\omega)\subset S(\omega)$ almost 
  surely for every $\alpha\in I$.  \par

  Furthermore, there exists a countable subset $I_0 \subset I$ such that 
  \begin{equation}\label{Verein}
     A(\omega)=\overline{\bigcup_{\alpha \in I_0}K_\alpha(\omega)}
  \end{equation} 
  up to a set of measure zero.
\end{prop}

\begin{proof} We can and will assume without loss of generality that 
  the family $K_\alpha$, $\alpha\in I$, is closed under finite unions. 
  For an arbitrary set $G \in \EE$ and $\alpha\in I$ we define
  \begin{equation*} 
    M_\alpha^{G}:=\{\omega\in\Omega:K_\alpha(\omega)\cap G \neq \emptyset\}.
  \end{equation*}
  Since~$E$ is Polish, $M_\alpha^{G}\in\mathscr F^u$ by the projection 
  theorem~\cite[III.44]{DM78}. 
  We can extend the probability measure~$\P$ to $\mathscr F^u$ in a 
  unique way (we will use the same notation for the extension). 
  Note that on the set $M_\alpha^{G}$ each~$S$ as in the lemma necessarily 
  satisfies $S(\omega) \cap G \neq \emptyset$ almost surely and therefore 
  the same must be true for~$A$ (which we yet need to define). 
  Let $\beta_{\alpha,G}:=\P\bigl(M_\alpha^{G}\bigr)$. 
  Then, there exists a sequence $\alpha_j=\alpha_j(G),j=1,2,\ldots$ in~$I$ 
  such that 
  $\P\bigl(M_{\alpha_j}^{G}\bigr)\nearrow \beta_G:=\sup_{\alpha\in I}\beta_{\alpha,G}$. 
  Define 
  \begin{equation*}
    M^{G}:=\bigcup_{j=1}^\infty  M_{\alpha_j}^{G}. 
  \end{equation*}
The set~$M^{G}$ depends on the choice of the sequence $\bigl(\alpha_j\bigr)$, 
but different choices give sets which only differ by a set of measure~$0$. 
Changing the sets $M_\alpha^{G}$ on a set of $\P$-measure 0, we can and will 
in fact assume that all these sets (and also the sets~$M^{G}$)  are even 
in~$\mathscr F$.
Note that for each $\alpha\in I$ we have $M_\alpha^{G}\subset M^{G}$ up 
to a set of measure~$0$. 
Recall that~$D$ is a countable and dense subset of~$E$. 
Defining
\begin{equation*}
  C_r:=\bigcup_{x \in D}\bigl(\overline{B(x,r)}\times M^{B(x,r)}\bigr),\qquad 
  C:=\bigcap_{n\in\N} C_{1/n},
\end{equation*}
one may hope to be able to show that $A:=C$ has the required properties. 
While the measurability property clearly holds it is not clear 
that~$A(\omega)$ is closed. 
One might therefore take the closure of the right hand side in the 
definition of~$C_r$, but then measurability of~$A$ is not clear. 
We will therefore change the definition of~$C_r$ in such a way that 
it becomes a closed random set for each $r>0$. 
Then~$C$ will be a closed random set as well.  \par

Define $R(x,r)$ as in Lemma~\ref{coverbem} and put 
\begin{equation*}
  A_r:=\bigcup_{x \in D}\bigl(R(x,r) \times M^{R(x,r)}\bigr),\qquad 
  A:=\bigcap_{n\in\N} A_{1/n}.
\end{equation*}
Clearly, $A_r \in \EE \otimes\mathscr F$ and $A_r(\omega)$ is closed 
for each $\omega \in \Omega$ and each $r>0$. 
As in the proof of Proposition~\ref{singleset} it follows that~$A_r$ 
is even a closed random set and hence the same is true for~$A$. 
  
For~$S(\omega)$ as in the proposition and $\alpha \in I$ we have to show that 
$K_\alpha(\omega)\subset A(\omega)\subset S(\omega)$ for $\P$-almost all 
$\omega\in \Omega$.  \par

Fix $\alpha\in I$.  
For each $r>0$ we have 
\begin{equation*}
  K_\alpha(\omega) 
  \subset\Bigl(\bigcup_{x \in D}\bigl(R(x,r)\times M_\alpha^{R(x,r)}\bigr)\Bigr)(\omega)
  \subset\Bigl(\bigcup_{x \in D}\bigl(R(x,r)\times M^{R(x,r)}\bigr)\Bigr)(\omega)=A_r(\omega)
\end{equation*}
up to a set of measure~$0$ (which may depend on~$\alpha$). 
Hence, up to a set of measure~$0$, 
\begin{equation}\label{oben}
  K_\alpha(\omega)\subset A(\omega).  
\end{equation}
To finish the proof it suffices to show that 
\begin{equation}\label{Gleichheit}
  A(\omega)=
  \overline{\bigcup_{x\in D,n \in \N,j \in \N}K_{\alpha_j(R(x,1/n))}(\omega)} 
\end{equation}
up to a null set (observe that the right hand side is in $S(\omega)$ almost surely). 
Note that the inclusion ``$\supset$'' follows from~\eqref{oben} and the fact 
that~$A(\omega)$ is closed, so it remains to show the inclusion ``$\subset$''.  \par

For every $n \in \N$ and almost every $\omega \in \Omega$ we have 
\begin{equation*}
  A(\omega)\subset
  \bigcup_{x \in D,M^{R(x,1/n)}\ni \omega} R(x,{\textstyle\frac1n})
  \subset \bigcup_{x\in D,m \in \N,j \in \N}K^{2/n}_{\alpha_j(R(x,1/m))}(\omega),  
\end{equation*}
where the upper index~$2/n$ denotes the closed~$2/n$-neighbourhood of a set. 
Taking intersections over all $n\in\N$, \eqref{Gleichheit} follows and the 
proof of the proposition is complete. 
\end{proof}

\begin{bem}
The assertion of Proposition~\ref{lemma} becomes wrong if the word 
\emph{closed} is deleted. 
As an example take~$\Omega$ a singleton, $E=\R$, $I$ a non-measurable 
subset of~$\R$, and $K_\alpha=\{\alpha\}$ for $\alpha \in I$. 
There is no minimal measurable subset of~$\R$ which contains~$I$. 
\end{bem}

Our next goal is to clarify whether forward and strict invariance, 
respectively, of a set~$K$ are inherited by~$\hat K$ given by 
(the proof of) Proposition~\ref{singleset}.

\begin{lemma}\label{hat}
  \begin{enumerate}[(i)]
  \item If $K\subset E\times\Omega$ is forward invariant then 
    so is~$\hat K$. 
  \item If $K\subset E\times\Omega$ is strictly invariant, and 
    if~$\hat K(\omega)$ is compact for almost every $\omega\in\Omega$, 
    then~$\hat K$ is strictly invariant.
  \end{enumerate}
\end{lemma}

\begin{proof}
Let~$K$ be forward invariant, so for fixed $t \ge 0$ we have 
$\varphi(t,\omega)K(\omega)\subset K(\vartheta_t\omega)$. 
Here and in the following equalities and inclusions are meant up to 
sets of measure 0. 
Hence 
\begin{equation*}
  K(\omega)
  \subset(\varphi(t,\omega))^{-1}(K(\vartheta_t\omega))
  \subset (\varphi(t,\omega))^{-1}(\hat K(\vartheta_t\omega)). 
\end{equation*}
The right hand side is a random subset with closed $\omega$-sections 
since $\varphi(t,\omega)$ is continuous and therefore contains~$\hat K$, 
so~$\hat K$ is invariant.  \par

Next suppose that~$K$ is strictly invariant. 
Changing~$\hat K$ on a set of measure~$0$ if necessary we can and 
will assume that~$\hat K$ is a compact random set. 
For $t\ge 0$ fixed we have 
\begin{equation*}
  \varphi(t,\omega)\hat K(\omega)
  \supset \varphi(t,\omega)K(\omega)
  \supset K(\vartheta_t\omega). 
\end{equation*}
If the left hand side is a random set with closed $\omega$-sections, 
then it must contain $\hat K(\vartheta_t\omega)$ and strict invariance 
of~$\hat K$ follows. 
Using the representation~\eqref{hatK} of the set~$\hat K$, we see that 
if $\varphi(t,\cdot)(S)$, where $S=C \times F$ with~$C$ a deterministic 
compact subset of~$E$ and $F\in\FF$, is a random set, then the same is 
true for $\varphi(t,\cdot)(\hat K)$. 
Observe that the map $y \mapsto d(y,\varphi(t,\omega)(S(\omega)))$ is 
measurable for each $y \in E$ since 
$d(y,\varphi(t,\omega)(S(\omega)))=\infty$ for $\omega\notin F$, 
$d(y,\varphi(t,\omega)(S(\omega)))=d(y,\varphi(t,\omega)C)$ for 
$\omega \in F$ and $d(y,\varphi(t,\omega)C)<r$ iff 
$d(y,\varphi(t,\omega)x_i)<r$ for all $i$ where $(x_i)$ is a 
countable dense set in $C$ and $r>0$. 
This shows that $\varphi(t,\omega)\hat K(\omega)$ is a random set. 
To see that it has closed $\omega$-sections we use the fact that 
$\hat K(\omega)$ is compact and $\varphi(t,\omega)$ is continuous. 
This completes the proof of the lemma. 
\end{proof}

\textbf{Proof} [of Theorem~\ref{propo}]\textbf{.} 
We apply Proposition~\ref{lemma} to $K_B:=\hat\Omega_B$, $B\in\B$, 
where $\hat\Omega_B$ is constructed from $\Omega_B$ as in 
Proposition~\ref{singleset}, and obtain a smallest closed random 
set~$A$ containing~$\hat\Omega_B$ almost surely for every $B\in\B$. 
Clearly~$A$ is also the minimal closed random set~$A$ 
containing~$\Omega_B$ almost surely for each $B\in\B$. 
We claim that (a slight modification of) $A$ is the minimal pullback 
$\B$-attractor. 
By assumption, the set $\K(\omega)$ from the assertion of 
Theorem~\ref{propo} contains $\Omega_B(\omega)$  and, by 
Proposition~\ref{singleset}, also $\hat\Omega_B(\omega)$ almost 
surely for every $B\in\B$. 
By minimality, we have $A(\omega)\subset \K(\omega)$ almost surely. 
Since~$\K(\omega)$ is almost surely compact so is $A(\omega)$. 
Changing~$A$ on a set~$N$ of measure zero containing those~$\omega$ for 
which~$A(\omega)$ is not compact (e.g.\ by redefining $A(\omega)=\{e\}$ 
for $\omega\in N$ with some fixed $e\in E$) we can assume that~$A$ is 
a compact random set. 
Proposition~\ref{lemma} further implies that~$A$ can be represented as 
the closure of a countable union of sets $\hat\Omega_B$, $B\in\B$, 
almost surely. 
Since~$\Omega_B$ is strictly invariant by Lemma~\ref{tq2} and 
$\hat\Omega_B(\omega)$ is almost surely compact, $\hat\Omega_B$ 
is strictly invariant by Lemma~\ref{hat}, and so is~$A$ and the proof 
of Theorem~\ref{propo} is complete.  \qed

\begin{bem}
It is natural to ask if Theorem~\ref{propo} remains true if the set~$\K$ 
is not required to be $\EE \otimes \FF$-measurable (but all other 
assumptions hold). 
This is not the case in general. 
As an example, take $\Omega=[0,1]$ equipped with Lebesgue measure~$\P$ 
on the Borel sets, take $\vartheta=\mathrm{id}$, $E=\R$, and 
$\varphi=\mathrm{id}$ (with discrete or continuous time). 
Let $f:\Omega\to\R$ be a function whose graph $B\subset E\times\Omega$ 
is non-measurable and let $\B:=\{B\}$.  
Then $\Omega_B(\omega)=B(\omega)=\{f(\omega)\}$ for all~$\omega$, 
hence $\Omega_B=B$ and the assumptions of Theorem~\ref{propo} hold 
with~$\K=B$ except that~$\K$ is not a random set. 
\emph{If} a pullback-$\B$-attractor exists then it cannot possibly be 
contained in $\K$, so~(iii) of Theorem~\ref{propo} does not hold. 
One might hope that Theorem~\ref{propo} remains true if one replaces~$\K$ 
by~$\hat \K$ in part~(iii). 
However, also this fails to hold true since $\K\subset E\times\Omega$ 
with~$\K(\omega)$ compact for almost all~$\omega$ does not guarantee 
that~$\hat\K(\omega)$ is almost surely compact and if we 
choose~$B=\K$ then no $\B$-pullback attractor exists.  \par

If we impose additional measurability assumptions on~$\B$ then 
measurability of~$\K$ is not required for Theorem~\ref{propo} to hold. 
Indeed, if~$\Omega_B$ is a random set for every $B\in\B$ (or even just 
$\Omega_B\in \EE\otimes\FF^\mathrm u$), then define~$A$ as in 
Proposition~\ref{lemma}. 
Due to formula~\eqref{Verein} we see that $A(\omega)\subset\K(\omega)$ 
almost surely, so all assumptions of Theorem~\ref{propo} hold with~$\K$ 
replaced by~$A$. 
\end{bem}

\section{Existence of a minimal weak attractor for $\mathcal B$} 

Now we discuss the corresponding question for weak attractors. 
We continue to use covers $R(x,r)$, $x\in D$, as in the previous 
section.  \par

\begin{thm}\label{weakthm} Let~$\B$ be an arbitrary family of sets 
$B\subset E\times\Omega$ and let $(\varphi,\vartheta)$ be an RDS 
on~$E$. 
\begin{enumerate}[(i)]
\item
  Assume that there exists a random set $\K\subset E\times\Omega$ such 
  that $\K(\omega)$ is $\P$-a.s.\ compact and $\K$ attracts $\B$, i.e.\
  \begin{equation*}
    \lim_{t\to\infty}d\bigl(\varphi(t,\vartheta_{-t}\omega)B(\vartheta_{-t}\omega), 
    \K(\omega)\bigr)=0\quad\mbox{in probability},
  \end{equation*}
  for every $B\in\B$. 
  Then  $(\varphi,\vartheta)$ admits at least one weak attractor for~$\B$. 
\item If the RDS~$(\varphi,\vartheta)$ has a weak attractor for~$\B$, 
  then it has a \emph{minimal} weak $\B$-attractor~$A$. 
  In addition, there is a countable sub-family~$\B_1\subset\B$ such 
  that~$A$ is also the minimal weak $\B_1$-attractor. 
\item Under the assumptions of~(i) the minimal weak $\B$-attractor~$A$ 
  satisfies $A(\omega)\subset\K(\omega)$ almost surely. 
\end{enumerate}
\end{thm}

\begin{proof} 
For~$G\in \EE$ and $B\in \B$ define 
\begin{equation*}
  \V_B^G:=\Big\{V\in{\FF}:\lim_{t\to \infty}
  \P^* \bigl(
  V\cap\{\omega:\varphi(t,\vartheta_{-t}\omega)(B(\vartheta_{-t}\omega))
   \cap G
   \neq\emptyset\}\bigr)=0\Big\}
\end{equation*}
and put $\beta:=\sup\{\P(V): V \in \V_B^G\}$. 
Choose an increasing sequence $V_i\in\V_B^G$ such that 
$\lim_{i\to\infty}\P(V_i)=\beta$. 
Then $V:=\bigcup_{i=1}^\infty V_i\in\V_B^G$ (since outer measures 
are monotone and subadditive). 
Put $M_B^{G}:=\Omega\setminus V$ and define 
\begin{equation*}
  K_B:=\bigcap_{n=1}^\infty
       \bigcup_{x\in D}
       \Bigl(R(x,{\textstyle\frac1n})\times M_B^{R(x,\frac 1n)}\Bigr).  
\end{equation*}

Fix~$B\in\B$. 
Then $K_B$ is a closed random set. 
We first show that~$K_B$ is a minimal weak $\{B\}$-attractor and that~$K_B$ 
is contained in $\K$. 

\emph{Step 1:} $K_B(\omega)\subseteq\K(\omega)$ almost surely 

Let $B\in\B$, $G\in\EE$ and denote 
$S_t:=\{\omega:\,\varphi(t,\vartheta_{-t}\omega)B(\vartheta_{-t}\omega)\cap G\neq\emptyset\}$ 
for $t \ge 0$ and 
$\M_\delta :=\{\omega:\,\K^\delta (\omega)\cap G \neq \emptyset\}$ for $\delta>0$. 
Then $\lim_{t\to\infty}\P^*(S_t\setminus\M_\delta)=0$ by assumption, 
whence $\M_\delta^c\in\V_B^G$ and therefore $M_B^G\subseteq\M_\delta $, 
so $\K^\delta\cap G\neq\emptyset$ on~$M_B^G$ almost surely for each 
$\delta\in\mathbb Q\cap(0,\infty)$. 
This implies 
\begin{equation*}
  \K\cap G\neq\emptyset\quad{on}\ M_B^G\ \mbox{a.s.}, 
\end{equation*}
provided~$G$ is closed (using the fact that~$\K(\omega)$ is almost 
surely compact). 
In particular, we have 
\begin{equation*}
  R(x,1/n)\subseteq\K^{3/n}(\omega)\ \mbox{on}\ M_B^{R(x,{\textstyle\frac1n})}\ \mbox{a.s.} 
\end{equation*}
and therefore
\begin{equation*}
  K_B(\omega)\subseteq\bigcap_{n=1}^\infty
  \bigcup_{x\in D}
  \Bigl(\K^{3/n}(\omega)\times M_B^{R(x,{\textstyle\frac1n})}\Bigr)(\omega)
    \subseteq \bigcap_{n=1}^\infty \K^{3/n}(\omega)=\K(\omega)\quad{a.s.}, 
\end{equation*}
proving Step 1.  \par\bigskip

Our next goal is to show that~$K_B$ attracts~$B$ in probability. 
Generally, we say (in this proof) that a random set $S\subset E\times\Omega$ 
\emph{attracts}~$B$ if 
\begin{equation*}
  \lim_{t\to\infty}d\bigl(\varphi(t,\vartheta_{-t}\omega)B(\vartheta_{-t}\omega), 
  S(\omega)\bigr)=0\quad\mbox{in probability}.
\end{equation*}

\emph{Step 2:} The following properties are easy to verify: 
\begin{itemize}
\item[a)] If $S_1$ and $S_2$ attract $B$, then so does $S_1 \cap S_2$.  
\item[b)] If $S_n$, $n\in\N$, are random sets such that $S_n(\omega)$ 
  is compact for almost all $\omega$ and such that~$S_n$ attract $B$, 
  $n\in\mathbb N$, then so does $\bigcap_{n=1}^\infty S_n$. 
\end{itemize}

\emph{Step 3:} We show that 
$E_n:=\bigcup_{x\in D}\Bigl(R(x,{\textstyle\frac1n})\times M_B^{R(x,\frac1n)}\Bigr)$ 
attracts~$B$ for each $n\in\N$.  \par

Let $\eps>0$ and let $K \subset E$ be compact such that 
$\P(\K(\omega)\subseteq K) \ge 1-\eps$. 
Let $D_0 \subset D$ be a finite set such that 
$K\subset\bigcup_{x\in D_0}\mathring R(x,{\textstyle\frac1n})$ and let 
$\delta:=\inf\{d(y,K):\, y \in \bigcup_{x \in D_0}
  \bigl(\mathring R(x,{\textstyle\frac1n})\bigr)^c\}$.  
Note that $\delta>0$. 
We have 
\begin{align*}
\lefteqn{\P^*\Bigl(
  d\bigl(\varphi(t,\vartheta_{-t}\omega)B(\vartheta_{-t}\omega),E_n(\omega)\bigr)\ge \eps
      \Bigr)}  \\
  &\le\P^*
  \Bigl(d\Bigl(
     \bigcup_{x \in D}\Bigl(\varphi(t,\vartheta_{-t}\omega)B(\vartheta_{-t}\omega)
       \cap R(x,{\textstyle\frac1n})\Bigr), 
     \bigcup_{x \in D_0} \Bigl(R(x,{\textstyle\frac1n})\times M_B^{R(x,\frac 1n)}\Bigr)
  \Bigr)
  \ge\eps\Bigr) \\
  &\le\P^*\bigl(\varphi(t,\vartheta_{-t}\omega)B(\vartheta_{-t}\omega)\nsubset K^\delta
           \bigr)  \\ 
  &\quad+\sum_{x\in D_0}
   \P^*\Bigl(d\Bigl(\varphi(t,\vartheta_{-t}\omega)B(\vartheta_{-t}\omega)\cap R(x,{\textstyle\frac1n})\Bigr),
         \Bigl(R(x,{\textstyle\frac1n})\times M_B^{R(x,\frac 1n)}\Bigr)
        \Bigr)
      \ge\eps\Bigr).
\end{align*}
By definition of $M_B^{R(x,\frac1n)}$ the sum converges to~$0$ as $t\to\infty$. 
Further, by the definition of $K$ and $\delta$, 
\begin{equation*}
  \limsup_{t\to\infty} 
  \P^*\bigl(\varphi(t,\vartheta_{-t}\omega)B(\vartheta_{-t}\omega)\nsubset K^\delta \bigr)<\eps. 
\end{equation*}
Since $\eps>0$ was arbitrary it follows that~$E_n$ attracts~$B$.  \par\bigskip

\emph{Step 4:} Define $S_n:=E_n\cap\K$. 
Then Steps~1, 2 and~3 and the definition of~$K_B$ show 
that~$K_B$ attracts~$B$.  \par\smallskip

Note that the statements in~Step 1 and Step~4 imply that~$K_B$ is a 
minimal closed random weak $B$-attracting set. 
Further we know that~$K_B(\omega)$ is compact for almost all 
$\omega\in\Omega$. 
Changing~$K_B(\omega)$ on a set of measure~$0$ we can ensure that~$K_B$ 
is a compact random set. 
In order to complete the proof that~$K_B$ is a minimal weak $\{B\}$-attractor 
it remains to show strict invariance of~$K_B$.  \par\medskip

\emph{Step 5:} $K_B$ is strictly invariant.  \par

The proof of this step is similar to that of Lemma~\ref{hat}. 
Fix~$t>0$. 
Since~$K_B$ attracts~$B$ in probability, we have 
\begin{equation}\label{erstegl}
  \lim_{s\to\infty}d\bigl(\varphi(t+s,\vartheta_{-s}\omega)B(\vartheta_{-s}\omega),
  K_B(\vartheta_t\omega)\bigr)=0\quad\mbox{in probability} 
\end{equation}
and 
\begin{equation}\label{zweitegl}
  \lim_{s\to\infty}d\bigl(\varphi(s,\vartheta_{-s}\omega)B(\vartheta_{-s}\omega),
  K_B(\omega)\bigr)=0\quad\mbox{in probability}. 
\end{equation}
Using the cocycle property and continuity of~$\varphi$ together with 
compactness of~$K_B$, \eqref{zweitegl} implies 
\begin{equation}\label{drittegl}
  \lim_{s\to\infty}d\bigl(\varphi(t+s,\vartheta_{-s}\omega)B(\vartheta_{-s}\omega),
  \varphi(t,\omega)K_B(\omega)\bigr)=0\quad\mbox{in probability}. 
\end{equation}
Using the fact that~$\varphi(t,\omega)K_B$ is a random set (as in 
the proof of Lemma~\ref{hat}), and using minimality of~$K_B$, we 
see from~\eqref{erstegl} and~\eqref{drittegl} that 
$K_B(\vartheta_t\omega)\subset \varphi(t,\omega)K_B(\omega)$ almost 
surely. 
Since~\eqref{erstegl} implies 
\begin{equation*}
  \lim_{s\to\infty}d\bigl(\varphi(s,\vartheta_{-s}\omega)B(\vartheta_{-s}\omega),
  \bigl(\varphi(t,\omega)\bigr)^{-1}K_B(\vartheta_t\omega)\bigr)=0\quad\mbox{in probability}, 
\end{equation*}
equation~\eqref{zweitegl} and minimality of~$K_B$ imply the converse inclusion.  
\par\bigskip

The rest of the proof of the theorem is identical to that in the 
pullback case: just replace the set~$\hat\Omega_B$ by~$K_B$. 
\end{proof}

\section{Non-existence of a minimal forward attractor}\label{non}

In this section we provide an example of an RDS which has a forward 
attractor, but which fails to have a \emph{smallest} forward attractor. 
Consider a stationary Ornstein-Uhlenbeck process~$Z$, i.e.\ 
a real-valued centered Gaussian process defined on~$\R$ with 
covariance $\E(Z(t)Z(s))=\frac12\exp(-|t-s|)$. 
We define~$Z$ on the canonical space~$C(\R,\R)$ of continuous 
functions from~$\R$ to~$\R$ together with the usual shift and 
equipped with the law of~$Z$. 
Then $Z(t,\omega)=Z(0,\vartheta_t\omega)$. 
Let $g:[0,1]\to[-1,0]$ be continuous, non-decreasing such 
that $g(1)=0$, $g(x)<0$ for all $x\in[0,1)$ and $g(x)=-1$ 
for all $x \in [0,1/2]$, and let $h(t,y)$ be the unique 
solution of the ordinary differential equation 
\begin{align*}
  \dot h(t)=
  \begin{cases}
    g(h(t)) &\qquad\mbox{if}\ h(t)>0  \\
    0 &\qquad\mbox{otherwise}
\end{cases}
\end{align*}
with initial condition $h(0)=y\in[0,1]$.  \par

Next we define, for $x\in\R$, $y\in[0,1]$, and $t\ge0$, 
\begin{align*}
  \varphi(t,\omega)(x,y):=
  \begin{cases}
    (x+Z(t,\omega)-Z(0,\omega),\,h(t,y)) &\mbox{if}\ t\le\tau(y)  \\[.7ex]
    \bigl(\mathrm e^{\tau(y)-t}(x-Z(0,\omega))+Z(t,\omega),0\bigr)
      &\mbox{if}\ t\ge\tau(y),
\end{cases}
\end{align*}
where $\tau(y):=\min\{s\ge 0: h(s,y)=0\}$. 
Note that the motion in the $y$-direction is deterministic and all 
points with $y>0$ move in parallel in the $x$-direction. 
If $y<1$, then after the finite (deterministic) time $\tau(y)$, 
the $y$-component arrives at~$0$ and stays there, while the first 
coordinate is attracted by the process $Z$ exponentially fast.  
It is straightforward to check that~$\varphi$ defines a continuous 
RDS on the Polish space $E:=\R\times[0,1]$. 
If we consider~$E$ equipped with the Euclidean metric then the 
singleton $\{(Z(0,\omega),0)\}$ is a forward attractor for the 
family $\B:=\{C\times\{0\}:C\subset\R\ \mbox{compact}\}$. 
This is no longer true if we change the metric on~$E$ in the 
following way (without changing the topology of $E$): 
\begin{equation*}
  d\bigl((x,y),(\tilde x,\tilde y)\bigr):=
  |\tilde y -y|+|\Gamma(\tilde x)-\Gamma(x)|, 
\end{equation*}
where~$\Gamma$ is strictly increasing, odd, continuous such that 
$\Gamma(x)=\exp\{\exp\{\exp (x)\}\}\}$ for large $x$ (the fact 
that this metric  works can be checked by using the fact that the 
running maximum of a stationary Ornstein Uhlenbeck process up to 
time $t$ is of the order $\sqrt{\log t}$). 
There are, however, many forward $\B$-attractors with respect to 
the metric~$d$, for example 
\begin{align*}
  A_\gamma(\omega):= 
  &\bigl([Z(0,\omega)-\gamma,Z(0,\omega)+\gamma]\times \{0\}\bigr) \\
  &\ \bigcup\Bigl(\bigl(\{Z(0,\omega)-\gamma\}\times [0,1]\bigr) 
  \cup\bigl(\{Z(0,\omega)+\gamma\}\times [0,1]\bigr)\Bigr)
\end{align*} 
for an arbitrary $\gamma>0$ (note that this set is strictly 
invariant!). 
Now if there would be a smallest forward $\B$-attractor $A(\omega)$ 
then~$A(\omega)$ would have to be contained in the intersection 
$A_1(\omega)\bigcap A_2(\omega)$, which is a subset of 
$\R\times\{0\}$. 
It is clear, however, that the set $\{(Z(0,\omega),0)\}$ is the 
only strictly invariant compact subset of $\R \times \{0\}$, 
and we already have noted that this is not a forward attractor. 
This contradicts the assumption that there is a smallest forward 
$\mathcal B$-attractor. 
Consequently, this RDS does not have a smallest $\B$-attractor.  \par

\section{Another example}
We construct an example of an RDS for which a minimal pullback point 
attractor~$\omega\mapsto A(\omega)$ exists which, however, does not 
coincide with $\overline{\bigcup_{x\in E} \Omega_x(\omega)}$ (writing 
$\Omega_x(\omega)$ instead of $\Omega_{\{x\}}(\omega)$ for brevity). 
This shows that Theorem~4 in Crauel and Kloeden~\cite{crauel-kloeden15} 
is not entirely correct. 
In the following example, $A(\omega)$ consists of a single point while 
$\bigcup_{x \in E} \Omega_x(\omega)$ coincides with the whole 
space~$E$ almost surely. 
Even though we have $\Omega_x(\omega)\subset A(\omega)$ for almost 
all $\omega\in\Omega$ (\cite[Theorem~3.4]{crauel01} or 
Remark~\ref{rem3}\,(iii)), the (uncountable) union of all 
$\Omega_x(\omega)$ will turn out to be considerably larger 
than~$A(\omega)$. 

\begin{bsp} Let $E:=S^1$ be the unit circle which we identify with 
the interval $[0,2\pi)$ equipped with the usual metric 
$d(x,y):=|x-y|\wedge (2\pi - |x-y|)$. 
Consider the SDE
\begin{equation*}
  \dd X(t)=\cos(X(t))\,\dd W_1(t) + \sin(X(t))\,\dd W_2(t)  
\end{equation*}
on~$E$, where~$W_1$ and~$W_2$ are independent standard Brownian motions. 
Then there exists a `stable point' $\omega\mapsto S(\omega)$, measurable 
with respect to ``the past'' $\sigma\{W(t):t\leq0\}$, $W=(W_1,W_2)$, which 
is the support of a random invariant measure, and whose Lyapunov exponent 
is negative (see Baxendale~\cite{baxendale}). 
The random one point set $\{S(\omega)\}$ is a (minimal) weak point 
attractor (even a forward point attractor) of the RDS~$\varphi$ which is 
generated by the SDE\@. 
Recall that the system is reversible. 
Reverting time and using the same argument for the time inverted 
system gives existence of an `unstable point' 
$\omega\mapsto U(\omega)$, measurable with respect to $\{W(t):t\geq0\}$ 
and therefore independent of~$S$, which is a weak point repeller. 
The domain of attraction of $\{S(\omega)\}$ is $E\setminus\{U(\omega)\}$ 
and that of $\{U(\omega)\}$ for the time-reverted flow is 
$E\setminus\{S(\omega)\}$. 

When considering the system with continuous time, 
$\omega\mapsto\{S(\omega)\}$ is \emph{not} a pullback point attractor 
of~$\varphi$, though. 
In fact we even have $\Omega_x(\omega)=E$ almost surely for each 
fixed $x \in E$ since for each fixed $y \in E$, the process 
$t\mapsto \varphi(-t,\omega)y$ 
is a Brownian motion on~$E$ and therefore hits~$x$ for (some) arbitrarily 
large values of~$t$ showing that $y \in \Omega_x(\omega)$ for almost all 
$\omega \in \Omega$. 
In particular, the unique pullback point attractor of~$\varphi$ is the 
whole space~$E$.  \par

To obtain the required example we therefore evaluate the RDS~$\varphi$ 
at integer times only, i.e.\ we define 
\begin{equation*}
  \psi_n(\omega,x):=\varphi_n(\omega,x),\;x \in E,\;n \in \N_0, 
\end{equation*}
and we work with $T=\mathbb Z$ instead of $T=\R$. 
We denote the restriction of $\vartheta$ to $T=\mathbb Z$ with the same 
symbol. 
We claim that now~$\{S(\omega)\}$ is a minimal pullback point attractor, 
but that the closure of the union over all $\Omega$-limit sets of the 
points in~$E$ equals~$E$ almost surely.  \par

To see the first claim, consider the neighbourhood 
$I=(S(\omega)-\eps,S(\omega)+\eps)$ (mod~$2\pi$) of~$S(\omega)$ 
for some $\eps\in(0,1)$. 
The claim follows once we show that for each $x \in E$ we have 
$\Omega_x(\omega)\subset I$ almost surely. 
Since the Lyapunov exponent of $\psi$ is negative, the normalized Lebesgue 
measure of the set $\psi_n^{-1}(\vartheta_{-n}\omega,\cdot)(I^c)$ converges 
to~$0$ geometrically fast as $n \to \infty$ almost surely. 
The first Borel-Cantelli lemma now implies that the Lebesgue measure of the 
set $C(\omega)$ of all $x \in E$ which are contained in the set 
$\psi_n^{-1}(\vartheta_{-n}\omega,\cdot)(I^c)$ for infinitely many $n\in \N$ 
is zero. 
Since the distribution of $C(\omega)$ is invariant under rotations of~$E$ 
this implies that for each fixed $x \in E$ we have $\P(x \in C(\omega))=0$. 
This being true for every~$\eps$ of the form~$1/m$ we obtain 
$\Omega_x(\omega)=\{S(\omega)\}$ almost surely.  \par

To see the second claim, take any non-empty (deterministic) compact interval 
$J \subset E$. 
Then the set $\psi_n^{-1}(\vartheta_{-n}\omega,\cdot)(J)$ is a non-trivial 
interval for each~$n$. 
Note that for the centre point $x \in J$, the process 
$n \mapsto \psi_n^{-1}(\vartheta_{-n}\omega,\cdot)(x)$ performs a random walk 
on~$E$ (Brownian motion evaluated at discrete time steps) and therefore, the 
set $\bigcup_n \psi_n^{-1}(\vartheta_{-n}\omega,\cdot)(x)$ is almost surely 
dense in~$E$. 
Moreover, for any $y \in E$ and any $\delta >0$, we almost surely find a 
sequence of integer random times~$(t_n)$ such that 
$[y-\delta,y+\delta]
  \cap\bigcap_n\psi_{t_n}^{-1}(\vartheta_{-t_n}\omega,\cdot)(J)
  \neq\emptyset$. 
Therefore, for any~$z$ in that set, we have 
\begin{equation*}
  \Omega_z(\omega)\cap J\neq\emptyset\quad\mbox{almost surely}, 
\end{equation*}
showing the second claim. 
Note that we have actually proved more than we claimed insofar that 
for every non-empty open subset $I$ of $E$ we have 
\begin{equation*}
  \overline{\bigcup_{z \in I} \Omega_z(\omega)}=E\quad
  \mbox{for}\ \P\mbox{-almost every}\ \omega. 
\end{equation*}
\end{bsp}

\section{Perfection}

\begin{prop}\label{perf1} 
Let $(\varphi,\vartheta)$ be an RDS as in Definition~\ref{tq1}. 
Then there exists an RDS~$\psi$ on the same measurable dynamical system 
and a set $\Omega_1 \in \mathscr{F}$ of measure 1 such that~$\psi$ agrees 
with $\varphi$ on $\Omega_1$ and~$\psi$ satisfies~(i) and~(ii) of 
Definition~\ref{tq1} without exceptional sets. 
\end{prop}

\begin{proof} 
By assumption, there exists a set $N\in\mathscr F$ such that for all 
$\omega \in N^c$ the following hold: 
$x \mapsto \varphi(t,\omega)x$ is continuous for every $t \in \T_2$, 
$\varphi(0,\omega)=\mathrm{id}$, and 
$\varphi(t+s,\omega)=\varphi(t,\vartheta_s\omega)\circ \varphi (s,\omega)$ 
for all $s,t \in \T_2$. 
Define 
\begin{equation*}
  \Omega_1:=\{\omega\in\Omega:
            \,\vartheta_s\omega\in N^c\ \mbox{for almost all}\ s\in \T_1\}
\end{equation*}
(here, ``almost all'' refers to Lebesgue measure in the continuous case 
and counting measure in the discrete case). 
Clearly, $\Omega_1$ has full measure  and is invariant 
under~$\vartheta_t$ for every $t \in \T_1$. 
Then $\psi(t,\omega)x:=\varphi(t,\omega)x$ in case $\omega\in\Omega_1$ and 
$\psi(t,\omega)=\mathrm{id}$ in case $\omega \notin \Omega_1$ satisfies the 
claims in the proposition. 
\end{proof}


\begin{thebibliography}{10}%
\bibliographystyle{abbrv}
\bibitem{arnold}
  L.~Arnold, 
  \emph{Random Dynamical Systems}, 
  Springer, New York~1998

\bibitem{arnold-scheutzow}
  L.~Arnold and M.~Scheutzow, 
  Perfect cocycles through stochastic differential equations,
  \emph{Probab.\ Theory Relat.\ Fields}~101 (1995) 65--88
  
\bibitem{baxendale}
  P.\,H.~Baxendale, 
  Asymptotic behaviour of stochastic flows of diffeomorphisms, 
  pp.~1--19 
  in 
  \emph{Stochastic processes and their applications (Nagoya, 1985)}, 
  Lecture Notes in Math., 1203, Springer, Berlin~1986 

\bibitem{CS04}
  I.~Chueshov and M.~Scheutzow, 
  On the structure of attractors and invariant measures for a class of monotone random systems, 
  \emph{Dyn.\ Syst.}, 19 (2004) 127--144 

\bibitem{crauel99}
  H.~Crauel, 
  Global random attractors are uniquely determined by 
  attracting deterministic compact sets,
  \emph{Ann.\ Mat.\ Pura Appl., IV.~Ser.}, 
  Vol.~CLXXVI (1999) 57--72

\bibitem{crauel01}
  H.~Crauel, 
  Random point attractors versus random set attractors, 
  \emph{J.~London Math.\ Soc., II.~Ser.}, 63 (2001) 413--427

\bibitem{crauel02}
  H.~Crauel, 
  \emph{Random Probability Measures on Polish Spaces}, 
  Series \emph{Stochastics Monographs}, Volume~11, 
  Taylor~\&~Francis, London and New York~2002  

\bibitem{crauel-flandoli}
  H.~Crauel and F.~Flandoli, 
  Attractors for random dynamical systems, 
  \emph{Probab.\ Theory Relat.\ Fields}~100 (1994) 365--393

\bibitem{crauel-kloeden15}
  H.~Crauel and P.\,E.~Kloeden, 
  Nonautonomous and random attractors, 
  \emph{Jahresber.\ Dtsch.\ Math.-Ver.} 117 (2015) 173--206 

\bibitem{crauel-dimitroff-scheutzow}
  H.~Crauel, G.~Dimitroff, and M.~Scheutzow, 
  Criteria for strong and weak random attractors, 
  \emph{J. Dynam.\ Differential Equations} 21 (2009) 233--247 

\bibitem{DM78} 
  C.~Dellacherie and P.A.~Meyer,
  \emph{Probabilities and Potential},
  North-Holland, Amsterdam~1978   

\bibitem{Du76}
  J.~Dugundji,
  \emph{Topology},
  Allyn and Bacon, Boston~1966 

\bibitem{Flandoli-Gess-Scheutzow1}
  F.~Flandoli, B.~Gess, and M.~Scheutzow, 
  Synchronization by noise, 
  \emph{Probab.\ Theory Relat.\ Fields} 168 (2017) 511-556


\bibitem{Flandoli-Gess-Scheutzow2}
  F.~Flandoli, B.~Gess, and M.~Scheutzow, 
  Synchronization by noise for order-preserving random dynamical systems, 
  \emph{Ann.\ Probab.} 45 (2017) 1325-1350 

\bibitem{Kager-Scheutzow} 
  G.~Kager and M.~Scheutzow, 
  Generation of one-sided random dynamical systems by stochastic differential equations, 
  \emph{Electronic J.\ Prob.}  2 (1997) 17 pages 

\bibitem{ochs}
  G.~Ochs, 
  Weak random attractors, 
  Report 449, Institut f\"ur Dynamische Systeme, Universit\"at 
  Bremen, 1999 
  
\bibitem{Scheutzow96}
  M.~Scheutzow, 
  On the perfection of crude cocycles,
  \emph{Random and Comp.\ Dynamics}~4 (1996) 235--255  

\bibitem{Scheutzow02}
  M.~Scheutzow, 
  Comparison of various concepts of a random attractor: 
  a case study. 
  \emph{Arch.\ Math.\ (Basel)}~78 (2002) 233--240  

\end{thebibliography}
\end{document}